\documentclass{amsart}
\usepackage[utf8]{inputenc}
\usepackage{amsmath}
\usepackage{mathtools}
\usepackage{amssymb}
\usepackage{caption}
\usepackage{enumitem}
\usepackage{amsthm}
\usepackage[hidelinks]{hyperref}
\usepackage{cleveref}
\crefname{lemma}{Lemma}{Lemmas}
\crefname{theorem}{Theorem}{Theorems}
\usepackage{xcolor}
\usepackage{comment}
\usepackage{graphicx}
\usepackage{tikz}
\makeatletter
\def\@settitle{\begin{center}%
  \baselineskip14\p@\relax
  \bfseries
  \uppercasenonmath\@title
  \@title
  \ifx\@subtitle\@empty\else
     \\[1ex]\uppercasenonmath\@subtitle
     \footnotesize\mdseries\@subtitle
  \fi
  \end{center}%
}
\def\subtitle#1{\gdef\@subtitle{#1}}
\def\@subtitle{}
\makeatother

\newtheorem{theorem}{Theorem}[section]
\newtheorem{corollary}[theorem]{Corollary}

\newtheorem{proposition}[theorem]{Proposition}

\theoremstyle{remark}
\newtheorem{remark}[theorem]{Remark}

\usepackage{chngcntr}
\usepackage{graphicx} 
\usepackage{float}
\counterwithout{equation}{section}
\counterwithout{theorem}{section}

\begin{document}

\title{Ping-pong in the projective plane over a nonarchimedean field}

\author{Sami Douba}
\address{Mathematisches Institut der Universit\"at Bonn, Endenicher Allee 60, 53115 Bonn, Germany}
\email{douba@math.uni-bonn.de}

\author{Dmitry Kubrak}
\address{Universit\'e Paris-Saclay, CNRS, Laboratoire de math\'ematiques d’Orsay, 91405 Orsay, France}
\email{dmkubrak@gmail.com}

\author{Konstantinos Tsouvalas}
\address{Max Planck Institute for Mathematics in the Sciences, Inselstraße 22, 04103 Leipzig, Germany}
\email{konstantinos.tsouvalas@mis.mpg.de}

\begin{abstract}
We show that any lattice in $\mathrm{SL}_3(k)$, where $k$ is a nonarchimedean local field, contains an undistorted subgroup isomorphic to the free product $\mathbb{Z}^2*\mathbb{Z}$. To our knowledge, the subgroups we construct give the first examples in the literature of  finitely generated discrete subgroups of nonarchimedean Lie groups that are not virtually isomorphic to lattices in such Lie groups.
Our result is in contrast to the case of $\mathrm{SL}_3(\mathbb{Z})$, in which the existence of a $\mathbb{Z}^2*\mathbb{Z}$ subgroup remains open.
\end{abstract}


\maketitle

Denote by $k$ a nonarchimedean local field.  The purpose of this note is to establish the following.

\begin{theorem}\label{thm:main}
Let $\Lambda$ be a lattice in $\mathrm{SL}_3(k)$, and let $\Delta'$ be a $\mathbb{Z}^2$ subgroup of $\Lambda$. Then there is a finite-index subgroup $\Delta$ of $\Delta'$ and an infinite-order element $g \in \Lambda$ such that the subgroup $\langle \Delta, g \rangle < \Lambda$ is undistorted and decomposes as the free product $\Delta * \langle g \rangle$.
\end{theorem}

\begin{remark}\label{rem:z2}
We remark that any lattice $\Lambda$ in $\mathrm{SL}_3(k)$ contains a $\mathbb{Z}^2$ subgroup. Indeed, by Margulis's arithmeticity theorem \cite{MR1090825}, any such $\Lambda$ is arithmetic, so that the existence of a $\mathbb{Z}^2$ subgroup of $\Lambda$ follows from \cite[Thm.~1(ii)]{MR1866442}.
\end{remark}

The subgroups we construct in the proof of Theorem \ref{thm:main} provide examples of finitely generated discrete subgroups of (the $k$-points of) a $k$-group, where $k$ is a nonarchimedean local field, that are not virtually isomorphic to a lattice in a $k'$-group for any local field $k'$ (see Remark \ref{rem:lattice}). To our knowledge, these are the first such examples in the literature. 

Consider the ring $\mathbb{F}_q[t]$ of polynomials with indeterminate $t$ over the finite field~$\mathbb{F}_q$ of order $q$, and let $\mathbb{F}_q((1/t))$ denote the completion of the function field $\mathbb{F}_q(t)$ with respect to the ``valuation at infinity.'' Since, by~\cite{zbMATH03382639}, the discrete subgroup $\mathrm{SL}_3(\mathbb{F}_q[t])$ of $\mathrm{SL}_3(\mathbb{F}_q((1/t)))$ is of finite covolume, one concludes from Theorem \ref{thm:main} (and Remark~\ref{rem:z2}) the following.

\begin{corollary}
There is a subgroup of $\mathrm{SL}_3(\mathbb{F}_q[t])$ isomorphic to $\mathbb{Z}^2*\mathbb{Z}$. 
\end{corollary}

This is in contrast to $\mathrm{SL}_3(\mathbb{Z})$, for which the existence of a $\mathbb{Z}^2*\mathbb{Z}$ subgroup remains open (see \cite[Prob.~3.3]{MR4556437}). By recent work of Dey and Hurtado \cite[Thm~1.4]{dey2024remarks}, a hypothetical $\mathbb{Z}^2*\mathbb{Z}$ subgroup of $\mathrm{SL}_3(\mathbb{Z})$ would necessarily act minimally on $\mathbb{P}(\mathbb{R}^3)$ (in fact, on the Furstenberg boundary of $\mathrm{SL}_3(\mathbb{R})$), so that no ping-pong argument of the form we use to establish Theorem \ref{thm:main} will apply to $\mathrm{SL}_3(\mathbb{Z})$. On the other hand, Soifer \cite{MR2999365} demonstrated the existence of discrete copies of $\mathbb{Z}^2 * \mathbb{Z}$ in $\mathrm{SL}_3(\mathbb{R})$, and indeed, it follows from Soifer's argument that many lattices in $\mathrm{SL}_3(\mathbb{R})$ contain $\mathbb{Z}^2 * \mathbb{Z}$ subgroups.

Before proceeding to the proof of Theorem \ref{thm:main}, we introduce some more notation and terminology. Let $\mathcal{O} \coloneqq \{\alpha \in k \> ; \> |\alpha| \leq 1\}$ and $\mathfrak{m} \coloneqq \{\alpha \in k \> ; \> |\alpha| < 1\}$, and let $\pi \in \mathfrak{m}$ be a uniformizer of $k$, i.e., a generator of the ideal $\mathfrak{m}$ in~$\mathcal{O}$. Let $\mathrm{val}_\pi : k^\times \rightarrow \mathbb{Z}$ be the normalized valuation of $k$ (so $\mathrm{val}_\pi(\pi) = 1$). Finally, let $p$ (respectively, $q$) be the characteristic (resp., order) of the residue class field~$\mathcal{O}/\mathfrak{m}$ of~$k$.

We denote by $\mathcal{F}$ the Furstenberg boundary of $\mathrm{SL}_3(k)$, viewed as the space of projective flags $(x,L)$, where $x \in \mathbb{P}(k^3)$ and $L$ is a projective line through $x$. We say an element $g \in \mathrm{SL}_3(k)$ is {\em regular} if $g$ is diagonalizable over $k$ and the absolute values of the eigenvalues of $g$ are all distinct. In this case, there is a unique flag $(x^+, L^+) \in \mathcal{F}$ (respectively, $(x^-, L^-) \in \mathcal{F}$), called the {\em attracting flag} (resp., {\em repelling flag}) of~$g$, such that~$g^n$ converges uniformly to the constant function $(x^\pm, L^\pm)$ on compact subsets of the set of all flags in $\mathcal{F}$ transverse to $(x^\mp, L^\mp)$ as $n \rightarrow \pm \infty$.

\begin{proof}[Proof~of~Theorem~\ref{thm:main}]
Since any discrete $\mathbb{Z}^2$ subgroup of $\mathrm{SL}_3(k)$ preserves and acts cellularly on an apartment in the Bruhat--Tits building of $\mathrm{SL}_3(k)$ (see \cite[Exercise~II.6.6(2)~and~Thm.~II.7.1]{bridson1999metric}), up to conjugating $\Lambda$ within $\mathrm{SL}_3(k)$, we have that a finite-index subgroup $\Delta$ of $\Delta'$ is generated by matrices 
\[
a := \begin{pmatrix} a_1 & & \\ & a_2& \\ & & a_3 \end{pmatrix},\>  b:= \begin{pmatrix} b_1 & & \\ & b_2 & \\ & & b_3 \end{pmatrix},
\]
where $|a_1| = |b_2| < 1$ and $|a_2| = |a_3| = |b_1| = |b_3|$. We now identify the affine chart $\{Z \neq 0\}$ of $\mathbb{P}(k^3) = \{[X:Y:Z] \> | \> X,Y,Z \in k\}$ with $k^2$ in the usual manner. This affine chart is preserved by $\Delta$ and the matrices $a$ and $b$ act on this affine chart via 
\[
\begin{pmatrix} \alpha_1 & \\ & \alpha_2 \end{pmatrix}, \begin{pmatrix} \beta_1 & \\ & \beta_2 \end{pmatrix},
\]
respectively, where $\alpha_i = \frac{a_i}{a_3}$ and $\beta_i = \frac{b_i}{b_3}$ for $i=1,2$. 
Up to replacing each of $a$ and~$b$ with its~$p(q-1)^{\text{st}}$ power, we can assume that each of $\alpha_1, \alpha_2, \beta_1, \beta_2$ is a multiple of a (possibly negative) power of $\pi^p$ by some element in $1 + \pi \mathfrak{m}$. 

Let $\mathcal{U} = (1+\pi\mathfrak{m}) \times (1+ \pi \mathfrak{m})\subset k^2$. For $x \in \mathcal{U}$, denote by $\mathcal{V}_x$ the union of $\mathcal{U}$ and all projective lines through $x$ that, when viewed as affine lines in $k^2$, have slope belonging to $\pi + \pi \mathfrak{m}$. We claim that $\mathcal{V}_x \cap \gamma \mathcal{U} = \emptyset$ for each $x \in \mathcal{U}$ and each nontrivial $\gamma \in \Delta$. 

We first explain in this paragraph how the claim completes the proof. Let~$\mathcal{W}$ be the subset of $\mathcal{F}$ consisting of all projective flags of the form $(x,L)$ where $x \in~\mathcal{U}$ and~$L$ is a line through $x$ of slope belonging to $\pi + \pi\mathfrak{m}$.  
Since $\mathcal{W}$ is a nonempty open subset of $\mathcal{F}$, there is a regular element $h \in \Lambda$ whose attracting flag $(x^+, L^+)$ and repelling flag $(x^-, L^-)$ are both contained in $\mathcal{W}$ (this follows again from \cite[Thm.~1(ii)]{MR1866442}, for instance).
Note that $\mathcal{V}_{x^\pm}$ is a neighborhood of $L^\pm$ in $\mathbb{P}(k^3)$. There is thus some positive integer $N_0$ such that for all $N \in \mathbb{Z}$ with $|N| \geq N_0$, we have ${h^N(\mathbb{P}(k^3) \smallsetminus \mathcal{V}_{x^\pm}) \subset \mathcal{U}}$. 
Setting $g:=h^{N_0}$, it now follows from a standard ping-pong argument that the subgroup $\langle \Delta, g \rangle < \Lambda$ decomposes as the free product $\Delta * \langle g \rangle$. Up to increasing~$N_0$, one can moreover ensure that the subgroup $\langle \Delta, g \rangle$ is undistorted in $\Lambda$; see Remark~\ref{undistorted}. 

We now prove the claim. Let $\gamma \in \Delta$ be nontrivial. It is clear that $\mathcal{U} \cap \gamma \mathcal{U} = \emptyset$. It thus suffices to show that for each $x,y \in \mathcal{U}$, the slope of the line joining $\gamma x$ to~$y$ is not in $\pi+ \pi\mathfrak{m}$. Write $x = (1+\lambda_1, 1+\lambda_2)$ and $y = (1+\mu_1, 1+\mu_2)$, where $\lambda_i, \mu_i \in \pi\mathfrak{m}$ for $i=1,2$, and let $(m,n) \in \mathbb{Z}^2 \smallsetminus \{(0,0)\}$. We want to show
\[
\frac{\alpha_2^m \beta_2^n(1+\lambda_2) - (1+\mu_2)}{\alpha_1^m \beta_1^n(1+\lambda_1) - (1+\mu_1)} \notin \pi + \pi\mathfrak{m}.
\]
Note that $\beta_1^n, \alpha_2^m \in 1 + \pi\mathfrak{m}$, so that, up to replacing each of the $\lambda_i$ with some other element of $\pi\mathfrak{m}$, it is enough to show
\[
\sigma: = \frac{\beta^n(1+\lambda_2) - (1+\mu_2)}{\alpha^m (1+\lambda_1) - (1+\mu_1)} \notin \pi + \pi\mathfrak{m},
\]
where $\alpha = \alpha_1$ and $\beta = \beta_2$. The latter is true since either $\sigma = \infty$ or $\mathrm{val}_\pi(\sigma) \neq 1$. 
\end{proof}

\begin{remark}\label{rem:lattice}
We justify that $\mathbb{Z}^2 * \mathbb{Z}$ does not embed as a lattice in the $k$-points of a $k$-group for any local field $k$. Indeed, suppose for a contradiction that $\mathbb{Z}^2 * \mathbb{Z}$ embeds as a lattice $\Gamma$ in ${\bf G}(k)$ for some $k$-group ${\bf G}$. Up to replacing $\Gamma$ with a finite-index subgroup, we may assume that ${\bf G}$ is $k$-connected. It follows from \cite[Thm.~5.2]{MR4057144} that~${\bf R}(k)$ is compact, where ${\bf R}$ denotes the radical of ${\bf G}$; since $\Gamma$ is torsion-free, up to replacing ${\bf G}$ with ${\bf G}/{\bf R}$, we may thus assume that ${\bf G}$ is semisimple. There are then almost $k$-simple $k$-subgroups ${\bf G}_1, \ldots, {\bf G}_n$ of ${\bf G}$ and an isogeny ${\bf G}_1 \times \cdots \times {\bf G}_n \rightarrow {\bf G}$. Assume the ${\bf G}_i$ are ordered such that ${\bf G}_i$ is $k$-anisotropic precisely for $i > m$, let ${\bf H} = {\bf G}_1 \times \cdots \times {\bf G}_m$, and let $\Gamma'$ be the lattice in ${\bf H}(k)$ obtained by first taking the pre-image of $\Gamma$ in ${\bf G}_1(k) \times \cdots \times {\bf G}_n(k)$, passing to a torsion-free finite-index subgroup, and then projecting to ${\bf H}(k)$. Since no finite-index subgroup of $\Gamma'$ splits as a direct product of two nontrivial groups, we must have that $\Gamma'$ is an irreducible lattice in~${\bf H}(k)$. Thus, by the Margulis normal subgroup theorem \cite[Thm.~IX.5.6]{MR1090825}, we must have that ${\bf H}$ is almost $k$-simple and of $k$-rank $1$ (for instance, since $\Gamma'$ has infinite abelianization). Since $\Gamma'$ is not Gromov-hyperbolic, we have that $\Gamma'$ cannot be cocompact in ${\bf H}(k)$, nor is it possible that $k$ is archimedean and ${\bf H}(k)$ is locally isomorphic to $\mathrm{SL}_2(\mathbb{R})$ as a real Lie group. In all remaining cases where~$k$ is archimedean, it follows from Prasad rigidity \cite{MR399351} that $\Gamma'$ cannot be a lattice in~${\bf H}(k)$, since for instance one can embed $\Gamma'$ as an infinite-covolume discrete subgroup of ${\bf H}(k)$. We conclude that $k$ is nonarchimedean, but then $\Gamma'$ cannot be finitely generated (see \cite[Cor.~3.13]{MR1794898}), a contradiction.
\end{remark}

\section{Undistorted free products}
Recall that a map $f: \mathcal{Y} \rightarrow \mathcal{X}$ between two metric spaces $(\mathcal{Y}, d_\mathcal{Y})$ and $(\mathcal{X}, d_\mathcal{X})$ is a {\em quasi-isometric embedding} if there is a constant $C > 1$ such that for all $y_1, y_2 \in \mathcal{Y}$,
\[
C \cdot d_\mathcal{Y}(y_1,y_2)+C \geq d_\mathcal{X}(f(y_1),f(y_2)) \geq \frac{1}{C}  d_\mathcal{Y}(y_1,y_2) - C.
\]

Fix an arbitrary local field $k$, and consider the $\ell^\infty$-norm $||\cdot||$ on $k^d$ given by $||\sum_{i}a_ie_i||=\sup_{i}|a_i|$, where $(e_1,\ldots,e_n)$ is the canonical basis of $k^d$. For $g\in \mathrm{GL}_d(k)$, the operator norm is defined as $||g||=\sup_{v\neq 0}\frac{||gv||}{||v||}$. Denote by $\mu:\mathrm{SL}_d(k)\rightarrow \mathbb{R}^d$ the Cartan projection. For more background on the definition of $\mu$ we refer the reader to \cite{MR514561} in the archimedean case and to \cite{MR327923} in the nonarchimedean case.

Let $\Gamma<\mathrm{SL}_d(k)$ be a finitely generated subgroup. Fix a word metric $|\cdot|:\Gamma \rightarrow \mathbb{R}_{+}$ on $\Gamma$ given by a finite generating set of $\Gamma$. Let $\mathcal{X}_d$ be the symmetric space or Bruhat--Tits building associated to $\mathrm{SL}_d(k)$. The subgroup $\Gamma<\mathrm{SL}_d(k)$ is said to be {\em quasi-isometrically embedded} if, for some (equivalently, any) $x \in \mathcal{X}_d$, the map $\Gamma \rightarrow \mathcal{X}_d$ given by $\gamma \mapsto \gamma x$ is a quasi-isometric embedding.\footnote{This condition does not depend on the choice of word metric $|\cdot|$.}
This condition is equivalent to the existence of constants 
$C,a>0$ such that for all $\gamma \in \Gamma$,
$$||\mu(\gamma)||\geq a|\gamma|-C,$$ where $||\mu(\gamma)||$ denotes the Euclidean norm of $\mu(\gamma)$. The latter condition is in turn equivalent to the existence of constants $\alpha,c>0$ such that for all $\gamma \in \Gamma$, $$||\gamma||\cdot ||\gamma^{-1}||\geq e^{\alpha |\gamma|-c}.$$

If $\Gamma$ is a quasi-isometrically embedded subgroup of $\mathrm{SL}_d(k)$ and $\Lambda$ is some finitely generated subgroup of $\mathrm{SL}_d(k)$ containing $\Gamma$, one has that $\Gamma$ is {\em undistorted} in $\Lambda$, that is, that the inclusion of $\Gamma$ in $\Lambda$ is a quasi-isometric embedding. Though we will not be needing this, we remark that it follows from a result of Lubotzky--Mozes--Raghunathan \cite{MR1828742} that if $\Lambda$ is a lattice in $\mathrm{SL}_d(k)$ and $d\geq 3$ then a finitely generated subgroup $\Gamma < \Lambda$ is undistorted in $\Lambda$ if and only if $\Gamma$ is quasi-isometrically embedded in $\mathrm{SL}_d(k)$. Recall also that if~$\Lambda$ is a lattice in $\mathrm{SL}_d(k)$ for $d \geq 3$, then $\Lambda$ possesses Kazhdan's property (T) and is thus finitely generated.

The following proposition is folklore.

\begin{proposition}\label{QI} Let $C_1,C_2\subset \mathbb{P}(k^d)$ be nonempty disjoint subsets, and $\Gamma_1,\Gamma_2$ be finitely generated infinite subgroups of $\mathrm{SL}_d(k)$ satisfying:
\begin{enumerate}[label=(\roman*), topsep=0pt,itemsep=-1ex,partopsep=1ex,parsep=1ex]
\item\label{pingpongsets} $\gamma_i C_j\subset C_i$ for $i\neq j$ and $\gamma_i \in \Gamma_i \smallsetminus \{1\}$, and
\item\label{uniformity} there exists $\varepsilon > 0$ such that $||\gamma_i v||\geq \varepsilon ||\gamma ||\cdot ||v||$ for every $[v]\in C_j$, $j\neq i$, and $\gamma_i \in \Gamma_i$.
\end{enumerate}
\medskip

\noindent Suppose further that $\Gamma_1,\Gamma_2<\mathrm{SL}_d(k)$ are quasi-isometrically embedded. Then there is a finite-index subgroup $\Gamma_2'<\Gamma_2$ such that $\langle \Gamma_1,\Gamma_2'\rangle<\mathrm{SL}_d(k)$ is quasi-isometrically embedded and decomposes as $\Gamma_1\ast \Gamma_2'$.\end{proposition} 

\begin{proof} We fix a word metric $|\cdot|$ on $\Gamma_i$ induced by some finite generating subset. By assumption, there are $c,\alpha_1>0$ such that for all $\gamma \in \Gamma_1$, \begin{align}\label{qie1}||\gamma||\cdot ||\gamma^{-1}||\geq e^{\alpha_1|\gamma|-c}.\end{align} Given the constants $c,\varepsilon>0$, we may choose $\alpha_2>0$ and a finite-index subgroup $\Gamma_2'<\Gamma_2$ with the property that, for all $\delta \in \Gamma_2' \smallsetminus \{1\}$, \begin{align}\label{qie2}||\delta||\cdot ||\delta^{-1}||\geq \varepsilon^{-4}e^{c}e^{\alpha_2|\delta|}.\end{align}

Now let $n \geq 2$ and suppose we have elements $\gamma_1, \ldots, \gamma_n$ belonging alternatingly to $\Gamma_1 \smallsetminus \{1\}$ and $\Gamma_2' \smallsetminus \{1\}$, and set $g:=\gamma_1\cdots \gamma_n$.
 Let $j\in \{1,2\}$ be such that $\gamma_n\in \Gamma_i$, let $j=3-i$, and choose $[v]\in C_j$. By conditions \ref{pingpongsets} and \ref{uniformity}, we have that $$||gv||\geq \varepsilon^n  ||\gamma_1||\cdots ||\gamma_n||\cdot ||v||,$$ hence $$||g||\geq  \varepsilon^n  ||\gamma_1||\cdots ||\gamma_n||.$$ By arguing similarly for $g^{-1}=\gamma_n^{-1}\cdots \gamma_1^{-1}$, we have that $$||g^{-1}||\geq  \varepsilon^n  ||\gamma_1^{-1}||\cdots ||\gamma_n^{-1}||.$$ We thus obtain \begin{align}\label{qie3}||g||\cdot ||g^{-1}||\geq  \varepsilon^{2n} \big(||\gamma_1||\cdot||\gamma_1^{-1}||\big)\cdots \big( ||\gamma_n||\cdot  ||\gamma_n^{-1}||\big).\end{align}

Since we assumed $\gamma_1, \ldots, \gamma_n$ belong alternatingly to $\Gamma_1 \smallsetminus \{1\}$ and $\Gamma_2' \smallsetminus \{1\}$, by using the estimates (\ref{qie1}), (\ref{qie2}), and (\ref{qie3}) and setting $\alpha:=\min\{\alpha_1,\alpha_2\}>0$, we obtain the bound \begin{align*}||g||\cdot ||g^{-1}||&\geq \varepsilon^{2n} e^{-c(\frac{n}{2}+1)}(\varepsilon^{-4}e^{c})^{\frac{n}{2}-1} e^{\alpha \sum_{i=1}^{n}|\gamma_i|}\\ &\geq  \varepsilon^{4} e^{-2c}e^{\alpha \sum_{i=1}^{n}|\gamma_i|}.\end{align*} This shows that the natural map $\Gamma_1 * \Gamma_2' \rightarrow \langle \Gamma_1,\Gamma_2'\rangle<\mathrm{SL}_d(k)$ is a quasi-isometric embedding, and is in particular injective (since $\Gamma_1 * \Gamma_2'$ has no nontrivial finite normal subgroups).\end{proof}

\begin{remark}\label{undistorted}
We explain how Proposition \ref{QI} implies that, in the proof of Theorem~\ref{thm:main}, we may choose $R > 0$ such that $\langle \Delta, g^R \rangle$ is undistorted and decomposes as ${\Delta * \langle g^R \rangle}$. Indeed, it is enough to show that for some $R' > 0$, the subgroups $\Gamma_1 := \Delta$, $\Gamma_2:= \langle g^{R'} \rangle$ of $\mathrm{SL}_3(k)$, and the subsets $C_1 := \mathbb{P}(k^3) \smallsetminus \mathcal{V}_{x^\pm}$, $C_2 := \mathcal{U}$ of $\mathbb{P}(k^3)$ satisfy the conditions of Proposition \ref{QI}. To that end, note first that, since $\mathcal{U}$ is a compact subset of~$\mathbb{P}(k^3)$ contained in the complement of the hyperplanes $\{X=0\}$, $\{Y=0\}$, and $\{Z=0\}$, there is some $\theta > 0$ such that for any unit vector $v = (v_1, v_2, v_3) \in k^3$ satisfying $[v] \in~\mathcal{U}$, we have $|v_i| \geq \theta$ for $i = 1,2, 3$. We then have for any $\gamma = \mathrm{diag}(a_1, a_2, a_3) \in~\Delta$ that   $$||\gamma v||=||(a_1v_1, a_2v_2, a_3v_3)||\geq \max_{1\leq i\leq 3}\theta |a_i| = \theta||\gamma||.$$ 

Since $\mathbb{P}(k^3) \smallsetminus \mathcal{V}_{x^\pm}$ is a compact subset of the complement of $L^+ \cup L^-$ in $\mathbb{P}(k^3)$, one can check (by diagonalizing $g$, for instance) that there exist $\theta' > 0$ and an integer $R' > 0$ such that $\|g^r v\| \geq \theta' \|g^r \| \cdot \|v\|$ for all $[v] \in \mathbb{P}(k^3) \smallsetminus \mathcal{V}_{x^\pm}$ and $r \in \mathbb{Z}$ with $|r| \geq R'$. One may now take $\epsilon:= \min\{\theta, \theta'\}$ in the statement of Proposition~\ref{QI}.

\end{remark}

\subsection*{Acknowledgements} S.D. was supported by the Huawei Young Talents Program. S.D. thanks Anna Wienhard for her invitation to the MPI MiS (Leipzig) in the fall of 2024, where this work was discussed. S.D. and D.K. also thank the IH\'ES for excellent working conditions. We thank Nic Brody for useful conversations and for his comments on an earlier draft of this note. We also thank an anonymous referee for instructive comments.

\bibliography{nonarchimedeanbib}{}
\bibliographystyle{siam}

\end{document}